\documentclass[11pt]{article}
\usepackage[utf8]{inputenc}  
\usepackage[T1]{fontenc}
\usepackage[english]{babel}
\usepackage{amsmath,textcomp,amssymb,geometry,graphicx,enumerate}
\usepackage{amscd}
\usepackage{algorithm} 
\usepackage[noend]{algpseudocode} 
\usepackage{hyperref}
\usepackage{graphicx}
\usepackage{stmaryrd}
\usepackage{tabularx}
\hypersetup{
    colorlinks=true,
    linkcolor=blue,
    filecolor=magenta,      
    urlcolor=blue,
}

\def\Name{Chen Song}  
\def\Login{University of Illinois at Chicago}
\def\Session{}
\newtheorem{lemma}{Lemma}[section]
\newtheorem{theorem}[lemma]{Theorem}
\newtheorem{conjecture}[lemma]{Conjecture}

\newtheorem{example}[lemma]{Example}
\newtheorem{remark}[lemma]{Remark}

\usepackage{amsmath}

\title{The Syzygy Matrix and the Differential for Rational Curves in Projective Space}
\author{\Name, \texttt{\Login}}
\date{}

\def\endproofmark{$\Box$}
\newenvironment{proof}{\par{\bf Proof}:}{\endproofmark\smallskip}

\begin{document}
\maketitle

\begin{abstract}
    In this paper, we study whether a given morphism $f$ from the tangent bundle of $\mathbb{P}^1$ to a balanced vector bundle of degree $(n+1)d$  is induced by the restriction of the tangent bundle $T_{\mathbb{P}^n}$ to a rational curve of degree $d$ in $\mathbb{P}^n$. We propose a conjecture on this problem based on Mathematica computations of some examples and provide computer-assisted proof of the conjecture for certain values of $n$ and $d$. 
\end{abstract}

\section{Introduction}
Rational curves are crucial in the study of birational geometry and the arithmetic of varieties.
Analyzing the restriction of vector bundles to a rational curve unveils important geometric properties
of both the vector bundles and the underlying base space. Let $C \xrightarrow{\phi} \mathbb{P}^n$ be a rational curve in projective space over a field $K$. By Grothendieck's theorem on vector bundles on $\mathbb{P}^1$, the restriction of the tangent bundle of $\mathbb{P}^n$ to C splits as a direct sum of line bundles $T_{\mathbb{P}^n}|_C= \bigoplus_{i=1}^{n} \mathcal{O}_{\mathbb{P}^1}(a_i)$.
We say that the restriction of tangent bundle is \emph{balanced} if $|a_i - a_j| \leq 1$ for every 
$i,j$. \par

The splitting type of the restriction of the tangent bundle and normal bundle of a rational curve is an extensively studied topic. In \cite{Ram90}, Ramella proved that the restriction of the tangent bundle of $\mathbb{P}^n$ to a general rational curve is balanced. In \cite{man}, Mandal studied the locus of morphisms from a rational curve to a Grassmannian with a specified splitting type of the restricted tangent bundle. In \cite{Sac}, Sacchiero proves that a general non-degenerate rational curve of degree $d$ in $\mathbb{P}^n$ has a balanced normal bundle when $n\geq 3$ and $char(K)=0$. In \cite{Lar}, Larson and Vogt extend this result to the case $char(K)=2$ .  In \cite{Coskun2024}, Coskun, Larson and Vogt proved that $N_{C/Gr(a,a+b)}$ is 2-balanced for a general rational curve in Grassmannian.  \par
Let $C$ be a general rational curve in $\mathbb{P}^n$ of degree $d\geq n+1$. Then $T_{\mathbb{P}^n}|_C$ has degree $(n+1)d$. By Ramella's result, we have a balanced restriction $$T_{\mathbb{P}^n}|_C= \mathcal{O}_{\mathbb{P}^1}(d+q)^{\oplus a} \bigoplus  \mathcal{O}_{\mathbb{P}^1}(d+q+1)^{\oplus b},$$ where $q=\lfloor \frac{d}{n} \rfloor$, $a=(q+1)n-d$ and $b=d-nq$. If we compose the tangent map $T_C \xrightarrow{d\phi} T_{\mathbb{P}^n}$ with the restriction map  $T_{\mathbb{P}^n} \rightarrow T_{\mathbb{P}^n}|_C$, we get a morphism $$f:\mathcal{O}_{\mathbb{P}^1}(2) \rightarrow \mathcal{O}_{\mathbb{P}^1}(d+q)^{\oplus a} \bigoplus  \mathcal{O}_{\mathbb{P}^1}(d+q+1)^{\oplus b}.$$ \par
Conversely, Eric Larson raised the following question. Given a morphism $$f:\mathcal{O}_{\mathbb{P}^1}(2) \rightarrow \mathcal{O}_{\mathbb{P}^1}(d+q)^{\oplus a} \bigoplus  \mathcal{O}_{\mathbb{P}^1}(d+q+1)^{\oplus b} $$ where $q$, $a$ and $b$ are defined above, when can we find a rational curve $C\subset \mathbb{P}^n$ such that $f$ is given by the composition of tangent map with restriction map?  \par 

 When $d<n$, the curve $C$ is degenerate and not all morphisms $f$ can be represented by $C$. The restriction is not balanced except when $d=1$. \par 
 When $d=n$, $C$ is a rational normal curve. $Hom(\mathcal{O}_{\mathbb{P}^1}(2) ,\mathcal{O}_{\mathbb{P}^1}(n+1)^{n}) \cong K^{n^2}$. Therefore, $f$ is induced by the restriction of the tangent map.
 \par 
When $d \geq n+1$, the question is still open.  Now we state our main conjecture and results on this problem. 

\begin{conjecture} \label{conj}
    Let $d\geq n+1$ be integers. Let $q=\lfloor \frac{d}{n} \rfloor$, $a=(q+1)n-d$, $b=d-nq$. \par
    (1) If $char(K)=0$, $n=2,3$ or $d \geq n+2$, then given a general morphism $$f:\mathcal{O}_{\mathbb{P}^1}(2) \rightarrow \mathcal{O}_{\mathbb{P}^1}(d+q)^{\oplus a} \bigoplus  \mathcal{O}_{\mathbb{P}^1}(d+q+1)^{\oplus b},$$ there is a rational curve $C \xrightarrow{\phi} \mathbb{P}^n$ of degree $d$ such that $f=\phi^* \circ d\phi $. \par
    (2) If $n \geq 4$ is an integer, then there exists a morphism 
$$f : \mathcal{O}_{\mathbb{P}^1}(2) \to \mathcal{O}_{\mathbb{P}^1}(n+1)^{\oplus (n-1)} \oplus \mathcal{O}_{\mathbb{P}^1}(n+3)$$
such that there is no rational curve 
$C \xrightarrow{\phi} \mathbb{P}^n$
of degree $n+1$ for which $f = d \phi \circ \phi^*$.

\end{conjecture}

As the main result of this paper, we provide computer assisted proofs of this conjecture for some numbers of $n$ and $d$. 

\begin{theorem}\label{thm1}
  Let $d\geq n+1$ be integers. $char(K)=0$.  $q=\lfloor \frac{d}{n} \rfloor$, $a=(q+1)n-d$, $b=d-nq$. 
Given a general morphism $$f:\mathcal{O}_{\mathbb{P}^1}(2) \rightarrow \mathcal{O}_{\mathbb{P}^1}(d+q)^{\oplus a} \bigoplus  \mathcal{O}_{\mathbb{P}^1}(d+q+1)^{\oplus b}$$, there is a rational curve $C \xrightarrow{\phi} \mathbb{P}^n$ of degree $d$ such that $f=\phi^* \circ d\phi $ for the following list of $n$ and $d$. 
\begin{itemize}
    \item $n=2$, $3 \leq d \leq 25$;
    \item $n=3$, $4\leq d \leq 17$; 
    \item $n=4$, $6 \leq d \leq 12$;
    \item $n=5$, $7 \leq d \leq 9$.
\end{itemize}
\end{theorem}

\begin{theorem}\label{thm2}
  If $4 \leq n \leq 8$, then there exists a morphism 
$$f : \mathcal{O}_{\mathbb{P}^1}(2) \to \mathcal{O}_{\mathbb{P}^1}(n+2)^{\oplus (n-1)} \oplus \mathcal{O}_{\mathbb{P}^1}(n+3)$$
such that there is no rational curve 
$C \xrightarrow{\phi} \mathbb{P}^n$
of degree $n+1$ for which $f = d \phi \circ \phi^*$.
\end{theorem}

\textbf{Organization of the paper.}  In Section \ref{pre}, we recall the preliminary facts needed in the rest of the paper. In Section \ref{rel} we  explicitly write down the relation between $f$ and syzygy matrix of the rational curve. In Section \ref{proof}, we show a computer assisted proof of Theorem \ref{thm1} and Theorem \ref{thm2}. 

\textbf{Acknowledgements.} I extend my deepest gratitude to my advisor, Izzet Coskun, for providing invaluable guidance and insightful suggestions throughout the course of my study at UIC. I would like to thank Ning Jiang for her constructive remarks on the proof of Theorem 1.3. I would also thank Yeqin Liu for many useful discussions. 

\section{Preliminaries} \label{pre}
In this sections, we collect necessary facts on rational curves and vector bundles on $\mathbb{P}^n$. \par
A \emph{rational curve} is a curve which is birationally equivalent to $\mathbb{P}^1$. A rational curve $C \subset \mathbb{P}^n$ can be described as a map $[s:t] \mapsto [G_0(s,t),...,G_n(s,t)]$ where $G_i$ are  homogeneous polynomials in variable $s,t$ of the same degree $d$ with no common zeros. Here $d$ is called the degree of the rational curve $C$. \par
Given a homogeneous ideal $I=(f_1,...,f_m) \subset K[x_0,...,x_n]$. The \emph{first syzygy module} $Syz(f_1,...,f_m)$ consists of $m$-tuples $(g_1,...g_m)$ of elements of $K[x_0,...,x_n]$ such that $f_1g_1+,,,+f_mg_m=0$. The syzygy module can be presented as the kernel of a map \( K[x_0,...,x_n]^m \to I \). To describe the syzygy module, we often use a \emph{syzygy matrix}. If we have relations
$ \sum_{j=1}^m l_{ij} f_j = 0 $
for \( i = 1, 2, \ldots, t \), then these relations can be organized into a matrix \( L \) such that
\[ L \cdot \begin{pmatrix} f_1 \\ f_2 \\ \vdots \\ f_m \end{pmatrix} = 0. \]
Here, \( L \) is a \( t \times m \) matrix with entries in \( K[x_0,...,x_n] \). \par
By the Birkhoff–Grothendieck theorem, every vector bundle $V$ of rank $r$ on $\mathbb{P}^1$ decomposes as a direct sum of line bundles 
\[
\mathcal{E} \cong \bigoplus_{i=1}^r \mathcal{O}_{\mathbb{P}^1}(a_i)
\]
for a unique sequence of integers $a_1\leq ...\leq a_r$. 
The sequence of integers $a_1,...,a_r$ is called the \emph{splitting type} of $V$. We say that the vector bundle is \emph{$j$-balanced} if $a_1-a_r\leq j$. If $j=0$, we say that $V$ is \emph{perfectly} balanced. If $j=1$, we say that $V$ is balanced. \par
In \cite{Ram90}, Ramella proved that the restriction of the tangent bundle of $\mathbb{P}^n$ to a general rational curve is balanced. Let $C$ be a general rational curve in $\mathbb{P}^n$ of degree $d\geq n+1$. Then $T_{\mathbb{P}^n}|_C$ has degree $(n+1)d$.  We have a balanced restriction $$T_{\mathbb{P}^n}|_C= \mathcal{O}_{\mathbb{P}^1}(d+q)^{\oplus a} \bigoplus  \mathcal{O}_{\mathbb{P}^1}(d+q+1)^{\oplus b},$$ where $q=\lfloor \frac{d}{n} \rfloor$, $a=(q+1)n-d$ and $b=d-nq$.

\section{Relation Between the Morphism and the Rational Curve} \label{rel}
In this part, we use the Euler sequence of projective space to set up a commutative diagram. This commutative diagram relates the syzygy matrix of the rational curve $C$ and the morphism $f$ in Conjecture 1.1 and Conjecture 1.2. We work over a field $K$ of general characteristic.  \par
 Let $n \geq 2$ be an integer. Consider the Euler sequence of projective space $\mathbb{P}^n$. 
\[
0 \to \mathcal{O}_{\mathbb{P}^n} \to \mathcal{O}_{\mathbb{P}^n}(1)^{\oplus (n+1)} \to T_{\mathbb{P}^n} \to 0
.\]
 Let $C \subset \mathbb{P}^n$ be a rational curve of degree $d$. Restricting the Euler sequence to $C$, we get
 \[
0 \to \mathcal{O}_{\mathbb{P}^1} \to \mathcal{O}_{\mathbb{P}^1}(d)^{\oplus (n+1)} \to T_{\mathbb{P}^n}|_C \to 0.
\]
Suppose $C$ is a general curve, by Ramella's result, we have 
\[
0 \to \mathcal{O}_{\mathbb{P}^1} \to \mathcal{O}_{\mathbb{P}^1}(d)^{\oplus (n+1)} \to \mathcal{O}_{\mathbb{P}^1}(d+q)^{\oplus a} \bigoplus  \mathcal{O}_{\mathbb{P}^1}(d+q+1)^{\oplus b} \to 0
\]
where $q=\lfloor \frac{d}{n} \rfloor$, $a=(q+1)n-d$, $b=d-nq$.
Combining this sequence with the Euler sequence of $\mathbb{P}^1$, we have the following commutative diagram. 
\[
\begin{CD}
0 @>>> \mathcal{O}_{\mathbb{P}^1} @>>> \mathcal{O}_{\mathbb{P}^1}(1)^{\oplus 2} @>{
\begin{pmatrix}
    -t& s
\end{pmatrix}
}>> \mathcal{O}_{\mathbb{P}^1}(2) @>>> 0 \\
    @. @. @V{J}VV @V{FH}VV @. \\
0 @>>> \mathcal{O}_{\mathbb{P}^1} @>{G}>> \mathcal{O}_{\mathbb{P}^1}(d)^{\oplus (n+1)}  @>{LP}>> \mathcal{O}_{\mathbb{P}^1}(d+q)^{\oplus a} \bigoplus  \mathcal{O}_{\mathbb{P}^1}(d+q+1)^{\oplus b} @>{}>> 0 \\
\end{CD}
\]
\par
Let $s$, $t$ be the coordinate variables of $\mathbb{P}^1$. Morphisms in this commutative diagram can be represented by matrices of homogeneous polynomials with variables $s$, $t$.\par
$G=(G_0,...,G_n)$ is defined by the embedding of the rational curve $C \subset \mathbb{P}^n$ where $$G_i=\sum_{j=0}^{d}g_{ij}s^jt^{d-j}$$ and $J$ is the Jacobian matrix of $G$. \\
$LP=\begin{pmatrix}
    L\\
    P
\end{pmatrix}$
is an $n\times (n+1)$ matrix representing the syzygy of the ideal generated by $G_1,...,G_n$. $L=\begin{pmatrix}
L_{10} & L_{01} & \cdots & L_{0n} \\
L_{20} & L_{11} & \cdots & L_{1n} \\
\vdots & \vdots & \ddots & \vdots \\
L_{a0} & L_{a1} & \cdots & L_{an}
\end{pmatrix}$ where $L_{ki}=\sum_{j=0}^ql_{kij}s^jt^{q-j}$ for $k=1,...,a$ and $i=0,...,n$. \\
$P=\begin{pmatrix}
P_{10} & P_{01} & \cdots & P_{0n} \\
P_{20} & P_{11} & \cdots & P_{1n} \\
\vdots & \vdots & \ddots & \vdots \\
P_{b0} & P_{b1} & \cdots & P_{bn}
\end{pmatrix}$ where $P_{ki}=\sum_{j=0}^{q+1}p_{kij}s^jt^{q+1-j}$ for $k=1,...,b$ and $i=0,...,n$. \\
$FH=(F|H)^{\intercal}$ is the matrix representing the morphism $f$ in Conjecture 1.1 and Conjecture 1.2. \\
$F=(F_1,...,F_a)$ where $F_i=\sum_{j=0}^{d+q-2}f_{ij}s^jt^{d+q-2-j}$. 
$H=(H_1,...,H_b)$ where $H_i=\sum_{j=0}^{d+q-1}h_{ij}s^jt^{d+q-1-j}$. 
Since this is a commutative diagram, we get the equation of matrices 
\begin{equation}
    LP.J=FH.
\begin{pmatrix}
    -t& s
\end{pmatrix}.\label{equ}
\end{equation}
Both sides of the equation (\ref{equ}) are $n\times 2$ matrices of homogeneous polynomials. If we can solve this system of equations, we will have a  polynomial relation between the curve $C$ and the morphism $f$.  \par
One natural idea is that we can represent the syzygy matrix $LP$ and the Jacobian matrix $J$ by $g_{ij}$ and solve (\ref{equ}) for $f_{ij}$ and $h_{ij}$ by $g_{ij}$. However, this computation is extremely difficult since it involves solving high degree polynomial equations with a huge number of variables.  \par
Instead, we choose to fix the syzygy matrix $LP$ and construct a rational curve whose ideal has this syzygy matrix. Formally speaking, we take $G_0,...,G_n$ to be $n\times n$-minors of $LP$. Then we compute the Jacobian $J$ and compare the coefficients of the polynomial in both sides of equation (\ref{equ}). It becomes a system of equations with variables $f_{ij}$, $h_{ij}$, $l_{kij}$ and $p_{kij}$.\par
The right hand side of \ref{equ} is $$FH.
\begin{pmatrix}
    -t& s
\end{pmatrix}=\begin{pmatrix}
-tF_1 & sF_1 \\
\vdots & \vdots \\
-tF_a & sF_a \\
-tH_1 & sH_1 \\
\vdots & \vdots \\
-tH_b & sH_b \\
\end{pmatrix}$$
Therefore, we can linearly solve for $f_{ij}$ and $h_{ij}$ by $l_{kij}$ and $p_{kij}$.
The number of variables $f_{ij}$ and $h_{ij}$ is $a(d+q-1)+b(d+q)$. The number of variables $l_{kij}$ and $p_{kij}$ is $((q+1)a+(q+2)b)(n+1)$. The difference is $n^2+2n$. This shows we have enough variables $l_{kij}$ and $p_{kij}$ to represent $f_{ij}$ and $h_{ij}$. \par
To conclude, we have the following theorems. 

\begin{theorem}\label{tthm}
Suppose $char(K) \gg 0$. Let $d\geq n+1$ be integers. $q=\lfloor \frac{d}{n} \rfloor$, $a=(q+1)n-d$, $b=d-nq$.  The previous argument gives a polynomial ring map $$\Psi: K^{((q+1)a+(q+2)b)(n+1)} \rightarrow K^{a(d+q-1)+b(d+q)}$$    from the space of syzygy matrix of rational curves of degree $d$ in $\mathbb{P}^n$ to $$\text{Hom}(\mathcal{O}_{\mathbb{P}^1}(2) , \mathcal{O}_{\mathbb{P}^1}(d+q)^a \oplus \mathcal{O}_{\mathbb{P}^1}(d+q+1)^b).$$
   Given a general morphism $$f:\mathcal{O}_{\mathbb{P}^1}(2) \rightarrow \mathcal{O}_{\mathbb{P}^1}(d+q)^a \bigoplus  \mathcal{O}_{\mathbb{P}^1}(d+q+1)^b,$$ there is a rational curve $C \xrightarrow{\phi} \mathbb{P}^n$ of degree $d$ such that $f=\phi^* \circ d\phi $ if and only if $\Psi$ is dominant. 
\end{theorem}
We will analyze the surjectivity of $\Psi$ in next section. 
In the following example, we use a given syzygy matrix to compute the map $FH$. 
\begin{example}
   Let $n=4$ and $d=5$. Given syzygy matrix $$LP=\begin{pmatrix}
-t & s & 0 & 0 & 0 & 0 \\
0 & -t & s-2t & 4s & 0 & 0 \\
0 & 0 & 0 & 0 & -t & s \\
-t^2 & -2t^2 & s^2 & 0 & 0 & 0 \\
\end{pmatrix}$$ 
We take $G_0,...,G_4$ to be $4\times 4$-minors of $LP$, we have 
$$G=\begin{pmatrix}
4s^5 & 4s^4t & 4s^2t^2(s + 2t) & 4st^4 & 4t^5
\end{pmatrix}.$$ 
The Jacobian matrix of $G$ is 
\[
J = \begin{pmatrix}
20s^4 & 0 \\
16s^3t & 4s^4 \\
4st^2(3s + 4t) & 8s^2t(s + 3t) \\
4t^4 & 16st^3 \\
0 & 20t^4
\end{pmatrix}
\]
Then we have 
$$LP.J=\begin{pmatrix}
-4s^4t & 4s^5 \\
-4st^2\left(s^2 + 2st + 4t^2\right) & 4s^2t\left(s^2 + 2st + 4t^2\right) \\
-4t^5 & 4st^4 \\
-8s^3t^2(s + 2t) & 8s^4t(s + 2t)
\end{pmatrix}$$

and 

\[
FH.
\begin{pmatrix}
    -t& s
\end{pmatrix}=\scalebox{0.75}{$
\begin{pmatrix}
-t^5 f_{10} - s t^4 f_{11} - s^2 t^3 f_{12} - s^3 t^2 f_{13} - s^4 t f_{14} & s t^4 f_{10} + s^2 t^3 f_{11} + s^3 t^2 f_{12} + s^4 t f_{13} + s^5 f_{14} \\
-t^5 f_{20} - s t^4 f_{21} - s^2 t^3 f_{22} - s^3 t^2 f_{23} - s^4 t f_{24} & s t^4 f_{20} + s^2 t^3 f_{21} + s^3 t^2 f_{22} + s^4 t f_{23} + s^5 f_{24} \\
-t^5 f_{30} - s t^4 f_{31} - s^2 t^3 f_{32} - s^3 t^2 f_{33} - s^4 t f_{34} & s t^4 f_{30} + s^2 t^3 f_{31} + s^3 t^2 f_{32} + s^4 t f_{33} + s^5 f_{34} \\
-t^6 h_{10} - s t^5 h_{11} - s^2 t^4 h_{12} - s^3 t^3 h_{13} - s^4 t^2 h_{14} - s^5 t h_{15} & s t^5 h_{10} + s^2 t^4 h_{11} + s^3 t^3 h_{12} + s^4 t^2 h_{13} + s^5 t h_{14} + s^6 h_{15}
\end{pmatrix}
$}
\]
If we set these two matrix equal, then we can solve for $FH$ as 
\par
$$FH=\begin{pmatrix}
4s^4 & 4s^3t + 8s^2t^2 + 16st^3 & 4t^4 & 8s^4t + 16s^3t^2
\end{pmatrix}.$$

\end{example}

In general, for each given pair of $n,d$, we can use Mathematica to symbolically compute $f_{ij}$ and $h_{ij}$. Here is the result for $n=2$ $d=4$. 

\begin{example}
    Let $n=2$ $d=4$. Suppose $char(K) \neq 2$. Solving Equation  (\ref{equ}) by Mathematica, we get the explicit expression of $\Psi$ represented by the elements of $FH$. 
    \begin{itemize}
    \item $f_{10} = (-l_{101} l_{120} + l_{100} l_{121}) l_{210} + l_{111} (l_{120} l_{200} - l_{100} l_{220}) + l_{110} (-l_{121} l_{200} + l_{101} l_{220})$
    \item $f_{11} = -l_{111} l_{120} l_{201} - 2 l_{102} l_{120} l_{210} + 2 l_{100} l_{121} l_{210} + 2 l_{100} l_{120} l_{211} + l_{112} (l_{120} l_{200} - l_{100} l_{220}) - l_{100} l_{111} l_{221} + l_{110} (-2 l_{122} l_{200} + l_{101} l_{220} + l_{100} l_{221})$
    \item $f_{12} = -2 l_{110} l_{122} l_{201} - l_{110} l_{121} l_{202} - 2 l_{100} l_{122} l_{211} - l_{101} l_{122} l_{211} + l_{112} (l_{121} l_{200} + 2 l_{100} l_{201} - l_{101} l_{220} - l_{100} l_{221}) + l_{101} l_{110} l_{222}$
    \item $f_{13} = -2 l_{110} l_{122} l_{202} - l_{102} l_{121} l_{211} - l_{101} l_{122} l_{211} + l_{112} (l_{121} l_{200} + 2 l_{100} l_{201} - l_{101} l_{220} - l_{100} l_{221}) + l_{101} l_{110} l_{222}$
    \item $f_{14} = -l_{112} (l_{120} l_{200} - l_{100} l_{220}) l_{121} l_{210} + l_{112} (l_{120} l_{220})$
    \item $f_{20} = l_{120} (-l_{201} l_{210} + l_{200} l_{211}) + l_{110} (l_{201} l_{220} - l_{200} l_{221}) + l_{100} (l_{211} l_{220} - l_{210} l_{221})$
    \item $f_{21} = l_{121} (-l_{201} l_{210} + l_{200} l_{211}) + l_{110} (-2 l_{202} l_{210} + 2 l_{200} l_{212}) + l_{101} (l_{211} l_{220} - l_{210} l_{221}) + 2 l_{100} l_{210} l_{222}$
    \item $f_{22} = -l_{120} l_{202} l_{211} + l_{122} (l_{201} l_{210} - l_{200} l_{211}) + l_{112} (-l_{201} l_{220} + 2 l_{200} l_{221}) + l_{111} (l_{201} l_{220} - l_{200} l_{221}) + l_{110} (2 l_{202} l_{210} - 2 l_{200} l_{212}) + l_{101} l_{120} l_{222} - 2 l_{100} l_{122} l_{221}$
    \item $f_{23} = l_{122} (-2 l_{202} l_{210} + 2 l_{200} l_{212}) + l_{121} (-l_{202} l_{211} + 2 l_{200} l_{212}) + l_{112} (-2 l_{201} l_{220} + 2 l_{200} l_{221}) - l_{111} (2 l_{202} l_{210} - 2 l_{200} l_{212}) + l_{101} l_{120} l_{222} - 2 l_{111} l_{210} l_{221} + 2 l_{110} l_{210} l_{222} + l_{100} l_{211} l_{222}$
    \item $f_{24} = l_{122} (-l_{202} l_{211} + l_{201} l_{212}) + l_{112} (l_{202} l_{221} - l_{201} l_{222}) + l_{111} (l_{201} l_{220} - l_{200} l_{221}) + l_{110} l_{202} l_{221} - l_{101} l_{212} l_{221} + l_{100} l_{211} l_{222}$
\end{itemize}
Notice that in this example $b=0$. Therefore, we do not have $H$ terms. If we substitute $f_{ij}$ by the result above, we will get the expression of $FH$. However, due to the complexity of the expression and the limited space available, we do not display it here.
\end{example}

\section{Proof of Main Theorems} \label{proof}
Now, we use Mathematica computation to prove Theorem \ref{thm1}. We restate it for convenience.
\begin{theorem}[Restatement of Theorem \ref{thm1}]
 Let $d\geq n+1$ be integers. $char(K)=0$.  $q=\lfloor \frac{d}{n} \rfloor$, $a=(q+1)n-d$, $b=d-nq$. 
Given a general morphism $$f:\mathcal{O}_{\mathbb{P}^1}(2) \rightarrow \mathcal{O}_{\mathbb{P}^1}(d+q)^{\oplus a} \bigoplus  \mathcal{O}_{\mathbb{P}^1}(d+q+1)^{\oplus b},$$ there is a rational curve $C \xrightarrow{\phi} \mathbb{P}^n$ of degree $d$ such that $f=\phi^* \circ d\phi $ for the following list of $n$ and $d$. 
\begin{itemize}
    \item $n=2$, $3 \leq d \leq 25$;
    \item $n=3$, $4\leq d \leq 17$; 
    \item $n=4$, $6 \leq d \leq 12$;
    \item $n=5$, $7 \leq d \leq 9$.
\end{itemize}
\end{theorem}

\begin{proof}
    By Theorem \ref{tthm}, we need to prove $\Psi$ is dominant. Since dominance is an open condition, it suffices to show the Jacobian matrix of $\Psi$ is of full rank at a random point. \par
    As showed in Section 3, $\Psi$ is represented by the matrix $FH$. We can use Mathematica to symbolically compute $FH$ by the syzygy matrix $LP$ of the rational curve. Taking the Jacobian $J'$ of $FH$, we get the matrix representation of the differential $d\Phi$. Now it suffices to show $J'$ has full rank at a point. We use the random number generator in Mathematica to create a random numerical matrix $LP$ and plug 
 $l_{kij}$ and $p_{kij}$ in the Jacobian $J'$. Finally, we use the built-in rank-checking function in Mathematica to compute the rank of $J'$. The result shows that $J'$ has full rank. Thus, $\Psi$ is surjective at an Zariski open set of $K^{((q+1)a+(q+2)b)(n+1)}$. Therefore, $\Psi$ is dominant. 
\end{proof}

To further explain this proof, we show the computation in detail for $n=2$, $d=3$. 

\begin{example}
    Let $n=2$. $d=3$. In this case, the ideal of $C$ is generated by 
    \begin{itemize}
    \item $G_0=-l_{121} p_{112} s^3 + l_{111} p_{122} s^3 - l_{121} p_{111} s^2 t + l_{120} p_{112} s^2 t + l_{111} p_{121} s^2 t + l_{110} p_{122} s^2 t - l_{121} p_{110} s t^2 - l_{120} p_{111} s t^2 + l_{111} p_{120} s t^2 + l_{110} p_{121} s t^2 - l_{120} p_{110} t^3 + l_{110} p_{120} t^3$
    \item $G_1=l_{121} p_{102} s^3 - l_{101} p_{122} s^3 + l_{121} p_{101} s^2 t - l_{111} p_{121} s^2 t - l_{101} p_{112} s^2 t + l_{120} p_{102} s^2 t - l_{110} p_{121} s^2 t - l_{100} p_{122} s^2 t + l_{120} p_{101} s t^2 + l_{111} p_{100} s t^2 - l_{101} p_{120} s t^2 + l_{100} p_{121} s t^2 + l_{120} p_{100} t^3 - l_{100} p_{120} t^3$
    \item $G_2=-l_{111} p_{102} s^3 + l_{101} p_{112} s^3 - l_{110} p_{102} s^2 t + l_{111} p_{101} s^2 t - l_{101} p_{111} s^2 t + l_{110} p_{112} s^2 t + l_{101} p_{110} s t^2 - l_{100} p_{111} s t^2 - l_{110} p_{100} t^3 + l_{100} p_{110} t^3$
\end{itemize}

    Solving Equation  (\ref{equ}) by Mathematica, we get the explicit expression of $\Psi$ represented by the elements of $FH$. \begin{itemize}
    \item $f_{1,0} = l_{111} l_{120} p_{100} - l_{110} l_{121} p_{100} - l_{100} l_{120} p_{110} + l_{100} l_{121} p_{110} + l_{101} l_{110} p_{120} - l_{100} l_{111} p_{120}$
    \item $f_{1,1} = l_{111} l_{120} p_{101} - l_{110} l_{121} p_{101} - l_{100} l_{121} p_{111} + l_{101} l_{110} p_{121} - l_{100} l_{111} p_{121}$
    \item $f_{1,2} = l_{111} l_{120} p_{102} - l_{110} l_{121} p_{102} - l_{100} l_{121} p_{112} + l_{101} l_{110} p_{122} - l_{100} l_{111} p_{122}$
    \item $h_{1,0} = -l_{120} p_{101} p_{110} + l_{120} p_{100} p_{111} + l_{100} p_{111} p_{120} - l_{100} p_{110} p_{121} + l_{100} p_{101} p_{122}$
    \item $h_{1,1} = -l_{121} p_{101} p_{110} - 2 l_{120} p_{102} p_{110} + 2 l_{121} p_{102} p_{111} + l_{111} p_{101} p_{112} + l_{100} p_{112} p_{120} - l_{111} p_{100} p_{122}$
    \item $h_{1,2} = -2 l_{121} p_{102} p_{110} - l_{120} p_{102} p_{111} + 2 l_{111} p_{102} p_{120} + l_{110} p_{102} p_{121} - l_{100} p_{112} p_{121} - 2 l_{111} p_{100} p_{122}$
    \item $h_{1,3} = -l_{121} p_{102} p_{112} + l_{111} p_{101} p_{122} - l_{101} p_{112} p_{121} - l_{111} p_{101} p_{122} + l_{101} p_{111} p_{122}$
\end{itemize}

Now we take the Jacobian matrix $J'$ of $FH$. Due to the limited width of the paper, we display the transpose \\
\resizebox{\textwidth}{!}
{$
\left(
\begin{array}{ccccc}
 l_{121} p_{110} - l_{111} p_{120} & l_{121} p_{111} - l_{111} p_{121} & l_{121} p_{112} - l_{111} p_{122} & p_{110} p_{121} - p_{111} p_{120} & \cdots \\
 l_{101} p_{120} - l_{121} p_{100} & l_{101} p_{121} - l_{121} p_{101} & l_{101} p_{122} - l_{121} p_{102} & p_{101} p_{120} - p_{100} p_{121} & \cdots \\
 l_{111} p_{100} - l_{101} p_{110} & l_{111} p_{101} - l_{101} p_{111} & l_{111} p_{102} - l_{101} p_{112} & p_{100} p_{111} - p_{101} p_{110} & \cdots \\
 l_{110} p_{120} - l_{120} p_{110} & l_{110} p_{121} - l_{120} p_{111} & l_{110} p_{122} - l_{120} p_{112} & 0 & \cdots \\
 l_{120} p_{100} - l_{100} p_{120} & l_{120} p_{101} - l_{100} p_{121} & l_{120} p_{102} - l_{100} p_{122} & 0 & \cdots \\
 l_{100} p_{110} - l_{110} p_{100} & l_{100} p_{111} - l_{110} p_{101} & l_{100} p_{112} - l_{110} p_{102} & 0 & \cdots \\
 l_{111} l_{120} - l_{110} l_{121} & 0 & 0 & l_{120} p_{111} - l_{110} p_{121} & \cdots \\
 l_{100} l_{121} - l_{101} l_{120} & 0 & 0 & l_{100} p_{121} - l_{120} p_{101} & \cdots \\
 l_{101} l_{110} - l_{100} l_{111} & 0 & 0 & l_{110} p_{101} - l_{100} p_{111} & \cdots \\
 0 & l_{111} l_{120} - l_{110} l_{121} & 0 & l_{110} p_{120} - l_{120} p_{110} & \cdots \\
 0 & l_{100} l_{121} - l_{101} l_{120} & 0 & l_{120} p_{100} - l_{100} p_{120} & \cdots \\
 0 & l_{101} l_{110} - l_{100} l_{111} & 0 & l_{100} p_{110} - l_{110} p_{100} & \cdots \\
 0 & 0 & l_{111} l_{120} - l_{110} l_{121} & 0 & \cdots \\
 0 & 0 & l_{100} l_{121} - l_{101} l_{120} & 0 & \cdots \\
 0 & 0 & l_{101} l_{110} - l_{100} l_{111} & 0 & \cdots \\
\end{array}
\right.
$}

\resizebox{\textwidth}{!}{
$
\left.
\begin{array}{cccc}
 \cdots & 2 p_{110} p_{122} - 2 p_{112} p_{120} & p_{111} p_{122} - p_{112} p_{121} & 0 \\
 \cdots & 2 p_{102} p_{120} - 2 p_{100} p_{122} & p_{102} p_{121} - p_{101} p_{122} & 0 \\
 \cdots & 2 p_{100} p_{112} - 2 p_{102} p_{110} & p_{101} p_{112} - p_{102} p_{111} & 0 \\
 \cdots & p_{110} p_{121} - p_{111} p_{120} & 2 p_{110} p_{122} - 2 p_{112} p_{120} & p_{111} p_{122} - p_{112} p_{121} \\
 \cdots & p_{101} p_{120} - p_{100} p_{121} & 2 p_{102} p_{120} - 2 p_{100} p_{122} & p_{102} p_{121} - p_{101} p_{122} \\
 \cdots & p_{100} p_{111} - p_{101} p_{110} & 2 p_{100} p_{112} - 2 p_{102} p_{110} & p_{101} p_{112} - p_{102} p_{111} \\
 \cdots & -2 l_{110} p_{122} - l_{111} p_{121} + 2 l_{120} p_{112} + l_{121} p_{111} & 2 l_{121} p_{112} - 2 l_{111} p_{122} & 0 \\
 \cdots & 2 l_{100} p_{122} + l_{101} p_{121} - 2 l_{120} p_{102} - l_{121} p_{101} & 2 l_{101} p_{122} - 2 l_{121} p_{102} & 0 \\
 \cdots & -2 l_{100} p_{112} - l_{101} p_{111} + 2 l_{110} p_{102} + l_{111} p_{101} & 2 l_{111} p_{102} - 2 l_{101} p_{112} & 0 \\
 \cdots & l_{111} p_{120} - l_{121} p_{110} & l_{120} p_{112} - l_{110} p_{122} & l_{121} p_{112} - l_{111} p_{122} \\
 \cdots & l_{121} p_{100} - l_{101} p_{120} & l_{100} p_{122} - l_{120} p_{102} & l_{101} p_{122} - l_{121} p_{102} \\
 \cdots & l_{101} p_{110} - l_{111} p_{100} & l_{110} p_{102} - l_{100} p_{112} & l_{111} p_{102} - l_{101} p_{112} \\
 \cdots & 2 l_{110} p_{120} - 2 l_{120} p_{110} & l_{110} p_{121} + 2 l_{111} p_{120} - l_{120} p_{111} - 2 l_{121} p_{110} & l_{111} p_{121} - l_{121} p_{111} \\
 \cdots & 2 l_{120} p_{100} - 2 l_{100} p_{120} & -l_{100} p_{121} - 2 l_{101} p_{120} + l_{120} p_{101} + 2 l_{121} p_{100} & l_{121} p_{101} - l_{101} p_{121} \\
 \cdots & 2 l_{100} p_{110} - 2 l_{110} p_{100} & l_{100} p_{111} + 2 l_{101} p_{110} - l_{110} p_{101} - 2 l_{111} p_{100} & l_{101} p_{111} - l_{111} p_{101} \\
\end{array}
\right)
$}

We use the random number generator in Mathematica to create a random numerical matrix $LP$ and plug 
 $l_{kij}$ and $p_{kij}$ in the Jacobian $J'$. Finally, we use the built-in rank-checking function in Mathematica to compute the rank of this numerical matrix and we get $7$. Hence, this matrix has full rank. Therefore, $\Psi$ is dominant.

\end{example}

\begin{remark}
    In fact, our Mathematica program can compute the Jacobian matrix of $\Psi$ in symbolic form for each given $n$ and $d$ . However, due to the memory limitation of our personal computer, we cannot compute the rank of this symbolic Jacobian matrix. Instead, we  numerically compute its rank at some random point. Consequently, we limit our results in characteristic $0$. For those with access to a workstation or high-performance computing resources, this program should yield results for the general characteristic case.
\end{remark}

Now we use Mathematica computation to prove Theorem \ref{thm2}. We restate it for convenience.
\begin{theorem}[Restatement of Theorem \ref{thm2}]

  If $n=4,5,6,7,8$, then there exists a morphism 
$f : \mathcal{O}_{\mathbb{P}^1}(2) \to \mathcal{O}_{\mathbb{P}^1}(n+2)^{\oplus (n-1)} \oplus \mathcal{O}_{\mathbb{P}^1}(n+3)$
such that there is no rational curve 
$C \xrightarrow{\phi} \mathbb{P}^n$
of degree $n+1$ for which $f = d \phi \circ \phi^*$.

\end{theorem}

\begin{proof}
    By Theorem \ref{tthm}, it suffices to prove $\Psi$ is not dominant. Notice that the argument in Theorem \ref{tthm} works  for a field $K$ of general characteristic. However, since non-surjectivity is a closed condition, if we can check $\Psi$ is not surjective for a general characteristic, then it is not surjective for all characteristic.  \par
    We use the built-in rank function in Mathematica to compute the rank of the matrix 
    \[
M=
\begin{pmatrix}
-tF_1 & sF_1 \\
\vdots & \vdots \\
-tF_{n-1} & sF_{n-1} \\
\end{pmatrix}.
\]
We find that for $n=4,5,6,7,8$, rank$(M)$=1. Hence there is linear relation between the rows of $M$. This shows there is first order relations in $FH$. Therefore, $\Psi$ is not surjective. 
\end{proof}

\bibliographystyle{plain}
\bibliography{tangent}

\begin{thebibliography}{1}

\bibitem{Coskun2024}
Izzet Coskun, Eric Larson, and Isabel Vogt.
\newblock Normal bundles of rational curves in grassmannians, 2024.

\bibitem{Cos18}
Izzet Coskun and Eric Riedl.
\newblock Normal bundles of rational curves in projective space.
\newblock {\em Mathematische Zeitschrift}, 288(3-4):803--827, 2018.

\bibitem{Mathematica}
Wolfram~Research{,} Inc.
\newblock Mathematica, {V}ersion 14.1.
\newblock Champaign, IL, 2024.

\bibitem{Lar}
Eric Larson and Isabel Vogt.
\newblock Interpolation for brill–noether curves.
\newblock {\em Forum of Mathematics, Pi}, 11:e25, 2023.

\bibitem{man}
Sayanta Mandal.
\newblock On the loci of morphisms from $\mathbb{P}^1$ to $g(r,n)$ with fixed splitting type of the restricted universal sub-bundle or quotient bundle, 2020.

\bibitem{Ram90}
L.~Ramella.
\newblock La stratification du schéma de hilbert des courbes rationnelles de $p^n$ par le fibré tangent restreint.
\newblock {\em C. R. Acad. Sci. Paris Sér. I Math.}, 311(3):181--184, 1990.

\bibitem{Sac}
Giovanni Sacchiero.
\newblock Fibrati normali di curve razionali dello spazio proiettivo.
\newblock {\em Annali di Matematica Pura ed Applicata}, 120:105--113, 1979.

\end{thebibliography}
\nocite{*}

\end{document}